\documentclass[12pt]{amsart}
\usepackage{amssymb}
\usepackage[margin=1in]{geometry}
\usepackage{bm}

\newtheorem{theorem}{Theorem}[section]
\newtheorem{corollary}[theorem]{Corollary}
\newtheorem{prop}[theorem]{Proposition}
\newtheorem{lemma}[theorem]{Lemma}

\theoremstyle{remark}
\newtheorem{remark}[theorem]{Remark}

\theoremstyle{definition}
\newtheorem{defn}[theorem]{Definition}
\newtheorem{example}[theorem]{Example}
\newtheorem{prob}[theorem]{Problem}

\newcommand{\C}{\mathbb C}
\newcommand{\R}{{\mathbb R}}

\newcommand{\lla}{\left\langle}
\newcommand{\rra}{\right\rangle}
\newcommand{\mc}{\mathcal}

\newcommand{\tn}{\textnormal}

\newcommand{\Z}{\mathbb Z}

\newcommand{\F}{\mathbb F}

\begin{document}
	
\title{On operator Connes-amenability of the Fourier-Stieltjes algebra}

\author{Volker Runde}
\address{Department of Mathematical and Statistical Sciences, University of Alberta, Edmonton, AB, Canada}
\email{vrunde@ualberta.ca}

\author{Nico Spronk}
\address{Department of Pure Mathematics, University of Waterloo, Waterloo,
	ON, Canada}
\email{nico.spronk@uwaterloo.ca}

\author{Matthew Wiersma}
\address{Department of Mathematics and Statistics, University of Winnipeg, Winnipeg, MB, Canada}
\email{m.wiersma@uwinnipeg.ca}

%\subjclass[2020]{Primary 43A30, 46J40; Secondary 46L07, 22D10}

\dedicatory{Dedicated to the memory of Zhong-Jin Ruan}

\begin{abstract}
	Runde and Spronk showed in 2004 that there are non-amenable groups $G$, including $\F_2$, {whose Fourier-Stieltjes algebra, $B(G)$,} is operator Connes-amenable. This result was surprising since the measure algebra $M(G)$ is Connes-amenable if and only if $G$ is amenable, which might lead one to guess that $B(G)$ should be operator Connes-amenable if and only if $G$ is amenable. This leads to the question: for which groups $G$ is $B(G)$ operator Connes-amenable? We make progress on this problem by {exhibiting} the first examples of groups {for which $B(G)$ is not operator Connes-amenable}. More specifically, we show that $B(G)$ is not operator Connes-amenable when $G$ is a non-compact locally compact group with property (T) and finite almost periodic compactification, or when $G$ is a discrete group without the factorization property.
\end{abstract}

\maketitle

\section{Introduction}

{Defined by the existence of a left invariant mean on $L^\infty(G)$, amenability of groups is one of the most important topics in the study of locally compact groups $G$. In his seminal 1972 paper \cite{J1}, Johnson defined a notion of amenability for Banach algebras by saying that a Banach algebra $A$ is \textit{amenable} if every bounded derivation into a dual $A$-bimodule is an inner derivation. He then proved that a locally compact group $G$ is amenable if and only if $L^1(G)$ is amenable. Amenable Banach algebras possess many desirable qualities, such as having bounded approximate identities, and, for C*-algebras, amenability coincides with the notion of nuclearity \cite{Co,Haag}. We refer the interested reader to the book \cite{Run-book} for further reading on amenability in the contexts of both groups and Banach algebras.}

{The Fourier algebra $A(G)$ and Fourier-Stieltjes algebra $B(G)$ of a locally compact group $G$ were introduced by Eymard in \cite{Ey}. The Fourier algebra $A(G)$ naturally identifies as the predual of the group von Neumann algebra $VN(G)$ and $B(G)$ as the dual space of the full group C*-algebra $C^*(G)$.
	As such, $A(G)$ and $B(G)$ are each operator spaces in addition to being Banach algebras.
	When $G$ is abelian, $A(G)$ and $B(G)$ can be identified with the group algebra $L^1(\widehat G)$ and measure algebra $M(\widehat G)$, respectively. This allows us to view $A(G)$ and $B(G)$ as dual objects to $L^1(G)$ and $M(G)$.}

{Due to the relationship between $L^1(G)$ and $A(G)$,} it is natural to conjecture that $A(G)$ should be amenable if and only if $G$ is amenable. Johnson showed, however, that this conjecture is false when he proved that $A(\tn{SU}(2))$ is non-amenable in 1994 (see \cite{J2}). This result was refined by Forrest and Runde in \cite{FR}, where they showed that $A(G)$ is an amenable Banach algebra if and only if $G$ contains an abelian subgroup of finite index.

Although the conjecture that $A(G)$ is amenable if and only if $G$ is amenable turned out to be false, Ruan showed that it was true in spirit when viewed through the correct lens. Following the development of abstract operator spaces, Ruan introduced the notion of operator amenability for completely contractive Banach algebras and showed in his 1995 paper \cite{Ruan} that $A(G)$ is operator amenable if and only if $G$ is amenable.

Connes introduced a notion of amenability for von Neumann algebras, which was later termed as Connes-amenability by Helemskii, that takes into account the weak*-topological structure of the von Neumann algebra in \cite{Co}. In his paper, Connes showed that a von Neumann algebra is Connes amenable if and only if it is injective. Runde later generalized this notion of Connes-amenability to the category of dual Banach algebras and showed the measure algebra $M(G)$ of a locally compact group $G$ is Connes-amenable if and only if $G$ is amenable.

The Fourier-Stieltjes algebra $B(G)$ of a locally compact group is readily seen as a dual object to $M(G)$ in a similar way as $A(G)$ is a dual object to $L^1(G)$, and it is generally recognized that both the weak* and operator space structures play a crucial role in understanding $B(G)$. As such, it is natural to conjecture that a locally compact group $G$ should be amenable if and only if $B(G)$ is ``operator Connes-amenable''. This notion of operator Connes-amenability was codified by Spronk and Runde in \cite{RunS}, where they showed that $B(G)$ is operator Connes-amenable whenever $G$ is an amenable locally compact group but also found that there are non-amenable groups $G$, such as free groups, for which $B(G)$ is also operator Connes-amenable. More precisely, it was proven that $B(G)$ is operator Connes-amenable whenever $G$ is a locally compact group where $C^*(G)$ is residually finite dimensional. This result is quite surprising and leads to the following question, which is the main topic of this paper.

\begin{prob}\label{prob:main}
	For which locally compact groups $G$ is $B(G)$ operator Connes-amenable?
\end{prob}

In the intervening years following Spronk and Runde's paper \cite{RunS}, no further progress on understanding Problem \ref{prob:main} has been made. It has been suggested by Dales that perhaps $B(G)$ might be operator Connes-amenable for every locally compact group \cite{Dales}. In this paper, we demonstrate $B(G)$ is not operator Connes-amenable for a large class of groups. In particular, we show that $B(G)$ is not operator Connes-amenable when $G$  is a non-compact locally compact group with property (T) and finite almost periodic compactification in section \ref{sec:LC}, or when $G$ is a discrete group without the factorization property in Section \ref{sec:discrete}.  In particular, $B(G)$ fails to be operator Connes-amenable for simple higher rank Lie groups and certain Kac-Moody groups. When $G$ is a discrete group, we are left with the question of whether the factorization property characterizes when $B(G)$ is operator-Connes-amenable. This is discussed in Section \ref{sec:end} along with several other open problems.

Although no further progress on Problem \ref{prob:main} has been made prior to this paper, there have been various other works studying stronger amenability properties of $B(G)$. Indeed, Runde and Uygul showed in 2015 that $B(G)$ is Connes-amenable if and only if $G$ contains an abelian subgroup of finite index \cite{RunU}, and Spronk showed in 2020 that if $G$ is connected, then $B(G)$ is operator amenable if and only if $G$ is compact \cite{Sp}. It is our belief, however, that these are incomplete notions of amenability for $B(G)$ since the first fails to account for $B(G)$'s operator space structure and the latter fails to account for $B(G)$'s weak* topology.

\section{Preliminaries}

\subsection{Operator Connes-amenability}

We introduce only the notions we will require in this subsection and refer the reader to the book \cite{Run-book} for a more comprehensive account on what is currently known on dual Banach algebras and Connes-amenability.

Recall that a \textit{dual Banach algebra} is a pair $(A,A_*)$ where $A$ is Banach algebra and $A_*$ is a predual of $A$ and multiplication in $A$ is separately $\sigma(A,A_*)$ continuous.

\begin{defn}[\cite{RunS}]
	Suppose that $A$ is a completely contractive dual Banach algebra. A dual
	completely bounded $A$-bimodule $E$ is \textit{normal} if the left and right module maps $A\to E$ are separately weak*-continuous.
\end{defn}

\begin{defn}[\cite{RunS}]
	A completely contractive dual Banach algebra $A$ is \textit{operator
		Connes-amenable} if every weak*-continuous, completely bounded derivation from $A$ into a
	normal, dual operator $A$-bimodule is inner.
\end{defn}

Just as amenable Banach algebras can be characterized in terms of a bounded approximate diagonal, operator Connes-amenability can also be characterized in terms of a type of diagonal. Before describing this characterization, we must define an ancillary space.

Given a dual Banach algebra $A$ and a Banach $A$-bimodule $E$, we let $C_{\sigma,w}(E)$ denote the set all elements $x\in E$ such that the module maps $A\to E$ given by $a\mapsto a\cdot x$ and $a\mapsto x\cdot a$ are weak*-weakly continuous.

Now suppose that $A$ is a completely contractive dual Banach algebra and let $m\colon A\widehat{\otimes} A\to A$ be the multiplication operator, where $\widehat\otimes$ denotes to operator space projective tensor product. Then $m^*$ maps $A_*$ into $C_{\sigma,w}((A\widehat{\otimes}A)^*)$. Consequently, $m^{**}$ induces a homomorphism $m_{\sigma,w}\colon C_{\sigma,w}((A\widehat{\otimes}A)^*)^*\to A$.

\begin{defn}
	Let $A$ be a dual, completely contractive Banach algebra. An element $\bm{D}\in C_{\sigma,w}((A\widehat{\otimes} A)^*)^*$ is an \textit{operator $C_{\sigma,w}$-virtual diagonal} if $a\cdot \bm{D}=\bm{D}\cdot a$ and $a (m_{\sigma,w}\bm D)=a$ for all $a\in A$.
\end{defn}

The following theorem is proven in \cite{Run2} for dual Banach algebras, but its proof also works for completely contractive dual Banach algebras when the obvious modifications are made.

\begin{theorem}
	Suppose $A$ is a completely contractive dual Banach algebra. Then $A$ is operator Connes-amenable if and only if $A$ admits a operator $C_{\sigma,w}$-virtual diagonal.
\end{theorem}

This implies, in particular, that if $A$ is operator Connes-amenable, then $A$ has a unit.

\subsection{Fourier-Stieltjes algebra}

We briefly review some basic facts about the Fourier-Stieltjes algebra, all of which are well known and can be found in \cite{Ey} or \cite{Ar}. See also the book \cite{Four-book}.

Suppose $G$ is a locally compact group. We let $\Sigma_G$ denote the class of all (SOT-continuous) unitary representations of $G$. Given $\pi\in \Sigma_G$ and $\xi,\eta\in \mc H_\pi$, we let $\pi_{\xi,\eta}\colon G\to \C$ denote the matrix coefficient defined by
$$ \pi_{\xi,\eta}(s)=\lla \pi(s)\xi,\eta\rra\qquad (s\in G).$$
The \textit{Fourier-Stieltjes algebra} is the set of all matrix coefficients
$$B(G)=\{\pi_{\xi,\eta} : \pi\in \Sigma_G,\ \xi,\eta\in \mc H_\pi\}.$$
Then $B(G)$ is a Banach algebra with respect to pointwise addition and multiplication, and norm
$$ \|u\|_{B(G)}=\inf\{\|\xi\|\|\eta\| : \pi\in \Sigma_G,\  \xi,\eta\in \mc H_\pi,\ u=\pi_{\xi,\eta}\}\qquad (u\in B(G)).$$
Furthermore, $B(G)=C^*(G)^*$ with respect to the dual pairing
$$\lla a,\pi_{\xi,\eta}\rra=\lla \pi(a)\xi,\eta\rra\qquad(a\in C^*(G),\ \pi\in \Sigma_G, \ \xi,\eta\in H_\pi).$$
This dual pairing imparts operator space and weak*-topological structures onto $B(G)$ with respect to which $B(G)$ is a completely contractive dual Banach algebra.

Given $\pi\in \Sigma_G$, we let $A_\pi$ and $B_\pi$ be the norm-closed and weak*-closed subspaces of $B(G)$ generated by
$$\{\pi_{\xi,\eta} : \xi,\eta\in H_\pi\},$$
respectively. Then $A_\pi$ and $B_\pi$ are invariant under left and right translation by $G$. Conversely, if $E$ is any nontrivial norm-closed subspace of $B(G)$ that is left and right translation invariant, then $E=A_\sigma$ for some $\sigma\in \Sigma_G$.
Furthermore,
$$B_\pi=\{\sigma_{\xi,\eta} : \sigma\in \Sigma_G,\ \sigma\text{ is weakly contained in }\pi,\ \xi,\eta\in\mc H_\pi\}.$$
The \textit{Fourier algebra} $A(G):=A_\lambda$, where $\lambda$ denotes the left regular representation of $G$, is then an ideal of $B(G)$.

Given $\pi\in \Sigma_G$, we consider the C*-algebra
$$C^*_\pi:=\overline{\tn{span}\{\pi(f) : f\in L^1(G)\}}^{\|\cdot\|}\subset B(\mc H_\pi).$$
This produces the reduced group C*-algebra when one considers $\pi=\lambda$ and the full group C*-algebra when takes $\pi$ to be the \textit{universal representation} of $G$:
$$\varpi=\oplus\{\sigma : \sigma\in\Sigma_G \text{ is a cyclic representation}\}.$$
Then $B_\pi$ is canonically the dual space of $C^*_\pi$ and
$$\lla \pi(f),u\rra=\int_G f(s)u(s)\,ds$$
for all $f\in L^1(G)$ and $u\in B_\pi$.

\section{Property (T) groups with finite almost periodic compactifications}\label{sec:LC}

This section will focus on groups $G$ where the almost periodic compactification $G^{\tn{AP}}$ is finite. When $G$ is connected, this reduces to the case when $G$ is \textit{minimally almost periodic}, i.e. $G^{\tn{AP}}$ is trivial. {The key technical condition on $G$ that we will require is that $G$ has property (T), i.e. that every unitary representation of $G$ with almost invariant vectors possesses invariant vectors. It is known that all higher rank semisimple Lie groups and their lattices possess property (T). Property (T) is a very useful tool since it can greatly restrict what types of actions a group can have on various spaces. We refer the reader to the book \cite{prop-T} for further reading on property (T).}

Let $G$ be a locally compact group. A unitary representation $\pi$ of $G$ is \textit{purely infinite} if $\pi$ contains no finite dimensional subrepresentations of $G$. The space $A_{\mc{PI}}(G)$ is defined to be the set of all matrix coefficients $\pi_{\xi,\eta}$ where $\pi$ is a purely infinite representation of $G$ and $\xi,\eta\in \mc H_\pi$. It is shown in \cite{RunS} that $A_{\mc PI}(G)$ is a norm-closed ideal of $B(G)$ and that
\begin{equation}\label{eqn:PI}
	B(G)=A_{\mc PI}(G)\oplus A(G^{\tn{AP}})\circ \iota
\end{equation}
where $\iota\colon G\to G^{\tn{AP}}$ is the canonical inclusion.

\begin{lemma}\label{lem:w*}
	Suppose $G$ is a locally compact group.
	\begin{enumerate}
		\item If $G$ has property (T), then $A_{\mc{PI}}(G)$ is weak*-closed in $B(G)$.
		\item If $G^{\tn{AP}}$ is finite and $A_{\mc{PI}}(G)$ is weak*-closed in $B(G)$, then $G$ has property (T).
	\end{enumerate}
\end{lemma}

\begin{proof}
	(1) Let $\pi\in \Sigma_G$ be chosen so that $A_{\mc{PI}}(G)=A_\pi$. Then $\pi$ is purely infinite. If $A_{\mc{PI}}(G)$ is not weak*-closed, then $\pi$ weakly contains some finite dimensional representation of $G$. In particular, some finite dimensional irreducible representation of $G$ is not an isolated point of $\widehat G$ in the Fell topology \cite[Lemma 1.2.4]{prop-T}, so $G$ does not have property (T) (see \cite[Theorem 1.2.5]{prop-T}).
	
	(2) Suppose that $G$ has finite almost periodic compactification and let $\widehat{G}_{\tn{f.d.}}$ denote the set of finite dimensional irreducible representations of $G$. Then $\widehat{G}_{\tn{f.d.}}$ is a discrete space when endowed with the relative Fell topology from $\widehat{G}$. Further supposing that $A_{\mc{PI}}(G)$ is weak*-closed, one obtains that $\widehat{G}_{\tn{f.d.}}$ is an open subset of $\widehat{G}$ by \cite[Theorem 1.2.5]{prop-T}. It follows that the trivial representation is isolated in $\widehat G$ and, hence, $G$ has property (T).
\end{proof}

Equipped with the previous lemma, we are ready to find our first class of groups $G$ where $B(G)$ is not operator Connes-amenable.

\begin{theorem}\label{thm:LC-main}
	Suppose $G$ is a non-compact locally compact group with property (T) and finite almost periodic compactification $G^{\tn{AP}}$. Then $B(G)$ is not operator Connes-amenable.
\end{theorem}

\begin{proof}
	Set $G'$ be the kernel of the group homomorphism $G\to G^{\tn{AP}}$. Then $G'$ has finite index inside of $G$. We claim that the $G'$ is minimally almost periodic. Indeed, if $\pi$ is any finite dimensional unitary representation of $G'$, then $\sigma:=\tn{Ind}_{G'}^G\pi$ is a finite dimensional unitary representation of $G$, hence $\sigma(g)=1$ for all $g\in G'$. It follows that $G'$ has trivial almost periodic compactification since $\pi$ is a subrepresentation of $\sigma|_{G'}$. So $B(G')=A_{\mc{PI}}(G')\oplus \C1_H$ by Equation \eqref{eqn:PI}.
	
	Note that $G'$ is non-compact and has property (T) since $G$ possesses these properties and $G'$ has finite index inside of $G$. Consequently, $A_{\mc{PI}}(G')$ is weak*-closed inside of $B(G')$ by Lemma \ref{lem:w*}. Further, $A_{\mc{PI}}(G')$ is not operator Connes-amenable since $A_{\mc{PI}}(G')$ is not unital. It follows that $B(G')=A_{\mc{PI}}(G')\oplus \C1_{G'}$ is not operator Connes-amenable either. Indeed, suppose that there exists a normal completely bounded dual $A_{\mc{PI}}(G')$-bimodule $E$ and a weak*-continuous, completely bounded derivation $D\colon A_{\mc{PI}}(G')\to E$ that is not inner. Then $E$ is naturally a normal completely bounded dual $B(G')$-bimodule with respect to the actions
	$$(u+\lambda 1)\cdot x=u\cdot x+\lambda x\quad\text{and}\quad x\cdot(u+\lambda 1)=x\cdot u+\lambda x$$
	for $u\in A_{\mc{PI}}(G')$, $\lambda\in \C, x\in E$,
	and $\widetilde{D}\colon B(G')\to E$ defined by
	$$\widetilde D(u+\lambda 1)=D(u)$$
	for $u\in u\in A_{\mc{PI}}(G')$, $\lambda\in \C$
	is a weak*-continuous, completely bounded derivation. Since $D$ is not an inner derivation, neither is $\widetilde D$. So $B(G')$ is not operator Connes-amenable.
	To deduce that $B(G)$ is not operator Connes-amenable, we simply note that the restriction map $B(G)\to B(G')$ is a weak*-continuous quotient map as $G'$ is an open subgroup of $G$.
\end{proof}

\begin{example}
	Suppose $G$ is a connected real Lie group and $R$ is the radical of $G$. Rothman showed in \cite{Roth} that $G$ is minimally almost periodic if and only if $G/R$ is semisimple without compact factors and $G=\overline{[G,G]}$. So if $G$ satisfies these conditions and has property (T), then $B(G)$ is not operator Connes-amenable. In particular, if $G$ is a non-compact simple Lie group with property (T), then $B(G)$ is not operator Connes-amenable.  This holds, for example, when $G=\tn{SL}(n,\R)$ for some $n\geq 3$ or, more generally, when $G$ is any higher rank real Lie group.
\end{example}

The previous theorem applies to certain discrete groups as well.

\begin{example}\label{ex:Gromov}
	Gromov showed that there exists uncountably many minimally almost periodic discrete torsion groups with property (T) in \cite{Grom}. If $G$ is such a group, then $B(G)$ is not operator Connes-amenable.
\end{example}

The examples of Gromov considered above can be criticized for being difficult to understand. Thankfully, Theorem \ref{thm:LC-main} can also be used to show that $B(G)$ fails to be operator Connes-amenable for certain discrete groups $G$ that can be concretely described. We briefly outline one such class of examples below.

\begin{example}
	Suppose that $G$ is a finitely generated infinite discrete group and that $G$ admits a (not necessarily unitary) finite dimensional representation $\pi\colon G\to M_n(\C)$ with $\ker \pi\neq G$. Then $\pi(G)$ is a residually finite group by Mal'cev's theorem, hence $G$ cannot be simple. Consequently, every infinite simple finite generated discrete group is minimally almost periodic. In particular, if $G$ is an infinite simple group discrete group with property (T), then $B(G)$ is not operator Connes amenable.

	Caprace and R\'emy show that certain non-affine Kac-Moody groups $G$ over finite ground fields satisfy the above hypotheses in \cite{CapR}. They also give the first example of a finitely presentable simple infinite group with property (T) in the same paper. We refer the interested reader to Caprace and R\'emy's paper for precise descriptions of such groups {and} note that these examples can be very explicitly described.
\end{example}

A particularly important class of locally compact groups that are minimally almost periodic are those non-compact groups $G$ that possess the Howe-Moore property. Recall that a locally compact group $G$ has the \textit{Howe-Moore property} if every unitary representation $\pi$ of $G$ either has nonzero invariant vectors or all matrix coefficients of $\pi$ vanish at infinity, i.e. $\pi_{\xi,\eta}\in C_0(G)$ for all $\xi,\eta\in \mc H_\pi$. Examples of groups with the Howe-Moore property include the following.
\begin{enumerate}
	\item Simple real Lie groups with finite centre (see \cite{HM}).
	\item Isotropic simple algebraic groups over non Archimedean local fields (see \cite{HM}).
	\item Closed subgroups $G$ of $\tn{Aut}\, T$ where $T$ is a semi-homogeneous tree where the degree of vertex is at least 3 and $G$ acts $2$-transitively on the boundary of $T$ (see \cite{BM}). 
	\item Strongly transitive, type preserving closed subgroups of $\tn{Aut}\,\Delta$, where $\Delta$ is a locally finite thick affine building (see \cite{C2}).
\end{enumerate}
See also \cite{C1} for a unified proof of (1), (2) and (3).

We include the following characterization of groups with the Howe-Moore property that may be of independent interest. It, in particular, shows that every group with the Howe-Moore {property} is minimally almost periodic, a fact that is surely known to experts.

Recall that the \textit{Rajchman algebra} of $G$ is $B_0(G):=B(G)\cap C_0(G)$.

\begin{prop}
	Let $G$ be a non-compact locally compact group. The following are equivalent.
	\begin{enumerate}
		\item $G$ has the Howe-Moore property,
		\item $B(G)=B_0(G)\oplus \C1_G$,
		\item $G$ is minimally almost periodic and $A_{\mc{PI}}(G)=B_0(G)$.
	\end{enumerate}
\end{prop}

\begin{proof}
	(1) $\Rightarrow$ (2): Suppose that $G$ has the Howe-Moore property and $u\in B(G)$. Then $u=\pi_{\xi,\eta}$ for some unitary representation $\pi$ and $\xi,\eta\in \mc H_\pi$. Let $\mc K$ denote the subspace of $G$-invariant vectors in $\mc H_\pi$ and let $P\in B(\mc H_\pi)$ be the orthogonal compliment onto $\mc K$. Then $u=\lambda 1_G+\pi_{P^{\perp}\xi,P^{\perp}\eta}$ where $\lambda=\lla P\xi,P\eta\rra$. Since $\pi(\cdot)|_{\mc K^{\perp}}$ has no nonzero invariant vectors, we deduce that $\pi_{P^{\perp}\xi,P^{\perp}\eta}\in B_0(G)$. The fact that $B_0(G)\cap \C1_G=\{0\}$ allows us to  conclude that $B(G)=B_0(G)\oplus \C1_G$.
	
	(2) $\Rightarrow$ (3): Suppose $B(G)=B_0(G)\oplus \C1_G$. As $B_0(G)\subset A_{\mc{PI}}(G)$, $\C1_G\subset A(G^{\tn{AP}})\circ\iota$ and $B(G)=A_{\mc{PI}}(G)\oplus A(G^{\tn{AP}})\circ\iota$, we deduce that $A_{\mc{PI}}(G)=B_0(G)$ and $A(G^{\tn{AP}})\circ\iota=\C1_G$.
	
	(3) $\Rightarrow$ (1): Suppose that $A_{\mc{PI}}(G)=B_0(G)$ and $G^{\tn{AP}}$ is trivial. Let $\pi$ be a unitary representation of $G$ that has no nontrivial invariant vectors. Then $\pi$ does not contain a copy of $1_G$. Since $1_G$ is the only finite dimensional irreducible representation of $G$, we deduce that $\pi$ is purely infinite. So $A_\pi\subset A_{\mc{PI}}(G)=B_0(G)\subset C_0(G)$.%Hence, $G$ is a Howe-Moore group.
\end{proof}

The fact that groups with the Howe-Moore property are minimally almost periodic allows us to give additional examples of groups where $B(G)$ fails to be operator Connes-amenable.

\begin{example}\ 
	\begin{enumerate}
		\item Let $G$ be an isotropic simple algebraic group over a non Archimedean local fields. If $G$ has property (T), then $B(G)$ is not operator Connes-amenable. This holds, for example, when $G=\tn{SL}(n,k)$ for some $n\geq 3$ where $k$ is a local field.
		\item Suppose $\Delta$ is a thick affine building of dimension at least 2 and $G$ is a strongly transitive, type preserving closed subgroup of $\tn{Aut}\,\Delta$. Then $G$ satisfies the Howe-Moore property by \cite{C2} and has property (T) by \cite{Op}. So $B(G)$ is not operator Connes-amenable.
	\end{enumerate}
\end{example}

\begin{example}
	Suppose that $G=\tn{SL}(n,\mathbb K)\ltimes \mathbb K^n$ where $n\geq 3$ and $\mathbb K$ is $\R$, $\C$ or a local field. Then $G$ has property (T) by \cite[Corollary 1.4.16]{prop-T} and we claim that $G$ is minimally almost periodic. Indeed, suppose that $\pi$ is a finite dimensional representation of $G$. Then $\pi(g)=1$ for all $g\in \tn{SL}(n,\mathbb K)$ since $\tn{SL}(n,\mathbb K)$ has the Howe-Moore property and, hence, is minimally almost periodic. Since the action of $\tn{SL}(n,\mathbb K)$ on $\mathbb K^n\setminus \{0\}$ is transitive, it follows that $\pi(v)=1$ for all $v\in \mathbb K^n$. So $\pi$ is an amplification of the trivial representation, showing that $G$ is minimally almost periodic. It follows that $B(G)$ is not operator Connes-amenable.
\end{example}

\section{Discrete Groups}\label{sec:discrete}

A locally compact group $G$ has the \textit{factorization property} if the unitary representation $\lambda\cdot \rho\colon G\to B(L^2(G))$ defined by $(\lambda\cdot \rho)(s)=\lambda(s)\rho(s)$ for $s\in G$ extends to $*$-representation of $C^*(G)\otimes_{\min}C^*(G)$. Here $\lambda$ and $\rho$ denote the left and right regular representations of $G$, respectively.
Every residually finite discrete group has the factorization property (e.g., see \cite[Proposition 7.3]{Oz}) and, conversely, every property (T) discrete group with the factorization property is residually finite \cite{k} (for a simple proof, see \cite[Theorem 7.4]{Oz}). Theorem {\ref{thm:discrete-main}} implies that if $G$ is discrete and $B(G)$ is operator Connes-amenable, then $G$ has the factorization property.

The following remark is well-known, but is included for the sake of self-containment.

\begin{remark}\label{rem:FP}
	Let $G$ be a discrete group and $\Delta$ be the diagonal subgroup of $G\times G$. Then $G$ has the factorization property if and only if the positive definite function $1_\Delta$ extends to a state on $C^*(G)\otimes_{\min}C^*(G)$. Indeed, this follows from the fact that $\lambda\cdot \rho$ is the GNS representation of $1_\Delta$ as $\delta_e\in \ell^2(G)$ is a cyclic vector for $\lambda\cdot \rho$ and $1_\Delta=(\lambda\cdot \rho)_{\delta_e,\delta_e}$.
\end{remark}

The main theorem of this section considers weak*-closed subalgebras of $B(G)$ that contain the Fourier algebra $A(G)$. The condition that such an algebra should contain $A(G)$ occurs very naturally in practice. Indeed, every complex conjugation invariant, point separating, weak*-closed, translation invariant subalgebra of $B(G)$ contains $A(G)$ \cite[Theorem 1.3.]{BLS}.

\begin{theorem}\label{thm:discrete-main}
	Let $G$ be a discrete group and $B$ a weak*-closed subalgebra of $B(G)$ that contains $A(G)$. If $B$ is operator Connes-amenable, then $G$ has the factorization property.
\end{theorem}

We provide two proofs of Theorem \ref{thm:discrete-main}. The first proof shows how operator Connes-amenability of $B(G)$ can be deduced from the existence of a $C_{\sigma,w}$-virtual diagonal and the second proof relies on a carefully chosen derivation. The latter proof is quicker whereas we consider the first proof to be more insightful.

\begin{proof}[Proof 1.]
	Fixing notation, we let $\varpi$ denote the universal representation of $G$ so that
	$$C^*(G)=C^*_\varpi=\overline{\tn{span}\{\varpi(s):s\in G\}}^{\|\cdot\|}.$$
	
	Let $B$ be a weak*-closed subalgebra of $B(G)$ containing $A(G)$. Then $B_*:=C^*(G)/B_\perp$ is a predual of $B$ where $B_\perp=\{a\in C^*(G) : \lla a,b\rra=0 \text{ for all }b\in B\}$. For each $s\in G$, we let $\theta(s)=\varpi(s)+B_\perp\in B_*$. Letting $m\colon B\widehat\otimes B\to B$ be the multiplication map, we have that
	$$m^*(\theta(s))=\theta(s)\otimes\theta(s)$$
	for all $s\in G$
	since if $u,v\in B$, then
	$$\lla \theta(s), m(u\otimes v)\rra=\lla \theta(s),uv\rra=u(s)v(s)=\lla \theta(s)\otimes \theta(s),u\otimes v\rra.$$
	
	Suppose $B$ admits a $C_{\sigma,w}$-diagonal $\bm D\in C_{\sigma,w}([B\widehat\otimes B]^*)^*$ and let $\psi=\bm D|_{B_*\check\otimes B_*}$. Then
	$$\lla \theta(s)\otimes\theta(s),\psi\rra=\lla m^*\theta(s),\psi\rra=\lla \theta(s),m_{\sigma,w} \bm D\rra=\lla \theta(s),1_G\rra=1$$
	for each $s\in G$. Also,
	$u\cdot \bm D=\bm D\cdot u$ for each $u\in B$ implies that
	$$u(s)\lla\theta(s)\otimes\theta(t),\psi\rra=\lla \theta(s)\otimes \theta(t), u\cdot \bm D\rra=\lla \theta(s)\otimes \theta(t), \bm D\cdot u\rra=u(t)\lla\theta(s)\otimes\theta(t),\psi\rra$$
	for all $u\in B$ and $s,t\in G$. In particular, taking $u=\delta_s\in A(G)\subset B$ implies that $\lla \theta(s)\otimes\theta(t),\psi\rra=0$ for all $s,t\in G$ with $s\neq t$.
	
	Let $q$ denote the canonical contractive map $C^*(G)\otimes_{\min}C^*(G)=C^*(G)\check\otimes C^*(G)\to B_*\check\otimes B_*$. Then $\psi\circ q\in (C^*(G)\otimes_{\min} C^*(G))^*$ and
	$$(\psi\circ q)(\varpi(s)\otimes\varpi(t))=1_\Delta(s,t)$$
	for all $s,t\in G$. In other words, $1_\Delta$ extends to a state on $C^*(G)\otimes_{\min} C^*(G)$, which shows that $G$ has the factorization property by Remark \ref{rem:FP}.
\end{proof}

\begin{proof}[Proof 2.]
	Let $B_{\min}(G\times G)=B_{\varpi\times \varpi}$ where $\varpi$ denotes the universal representation of $G$. Then $B_{\min}(G\times G)$ is the dual space of $C^*(G)\otimes_{\min}C^*(G)$. Set
	$$E=\{v\in B_{\min}(G\times G) : v(s,s)=0\text{ for all }s\in G\}.$$
	Then $E$ is a weak*-closed ideal of {$B_{\min}(G\times G)$} and, thus, is a normal completely bounded dual $B(G)$-bimodule. Consider the completely bounded, weak* continuous derivation $D\colon B\to E$ defined by
	$$D(u)=u\otimes 1-1\otimes u.$$
	Suppose that there exists $v\in E$ such that $D(u)(s,t)=u\cdot v-v\cdot u$, i.e.
	$$D(u)(s,t)=(u(s)-u(t))v(s,t)$$
	for all $s,t\in G$. Then
	$$D(\delta_s)(s,t)=v(s,t)=1$$
	for all $s,t\in G$ with $s\neq t$.
	So $v=1_{G\times G}-1_{\Delta}$. This implies that $1_{\Delta}\in B_{\min}(G\times G)$ and, thus, that $G$ has the factorization property by Remark \ref{rem:FP}.
\end{proof}

\begin{remark}
	Theorem \ref{thm:LC-main} is implied by Theorem \ref{thm:discrete-main} in {the} special case when $G$ is discrete. Indeed, suppose that $G$ is an infinite discrete group with property (T). If $G$ has finite almost periodic compactification, then $G$ is not residually finite and, hence, does not have have the factorization property. So the failure of $B(G)$ to be operator Connes-amenable follows from Theorem \ref{thm:discrete-main} when one takes $B=B(G)$.
\end{remark}

Theorem \ref{thm:discrete-main} allows us to describe more concrete examples of discrete groups $G$ where $B(G)$ is not operator Connes-amenable.

\begin{example}\label{ex:Sp}
	Suppose that $G$ is a lattice in the universal covering of $\tn{Sp}(n,1)$. Then, as observed in \cite[Remark 3.2]{Thom}, $G$ does not have the factorization property since $G$ has property (T) and work of Raghunathan shows that $G$ is not residually finite \cite{Rag}. Consequently, $B(G)$ is not operator Connes-amenable.
\end{example}

\begin{example}
	A group $G$ is \textit{Hopfian} if it is isomorphic to no proper quotient of itself. It is known that all finitely generated residually finite discrete groups are Hopfian \cite{Mal}. So if $G$ is a non-Hopfian discrete group with property (T), then $G$ does not have the factorization property and, hence, $B(G)$ is not operator Connes-amenable.
	
	Examples of non-Hopfian groups with property (T) are found in \cite{OllW}, \cite{Cor} and \cite{Thom}. The groups found in \cite{OllW} are constructed using a random process, whereas the examples in \cite{Cor} and \cite{Thom} are very concretely described as a quotient of a group of matrices over certain rings by a central subgroup.
\end{example}

We end this section by noting that the converse to Theorem \ref{thm:LC-main} is also true in the presence of Property (T). Recall first that a locally compact group is \textit{maximally almost periodic} if finite dimensional representations separate points of $G$. Of course, residually finite groups are examples of maximally almost periodic groups.

The following proof is similar to that of \cite[Theorem 4.6]{RunS}, which states that if $G$ is a locally compact group such that $C^*(G)$ is residually finite dimensional (such as $\F_2$), then $B(G)$ is operator Connes-amenable.

\begin{prop}\label{prop:converse}
	Let $G$ be a maximally almost periodic locally compact group. Then $B(G)$ admits a weak*-closed subalgebra $B$ containing $A(G)$ that is operator Connes-amenable.
\end{prop}

\begin{proof}
	Let $\pi=\oplus\{\sigma\in \widehat G : \pi\text{ is finite dimensional}\}$. Then $A_\pi$ is conjugate-invariant and point separating. So $B_\pi$ contains $A(G)$ by \cite{BLS}. Further, $A_\pi$ is operator amenable by \cite[Lemma 4.5]{RunS}. It follows that $B_\pi$ is operator Connes-amenable by weak*-density of $A_\pi$ inside of $B_\pi$.
\end{proof}

\begin{remark}
	If $G$ is a locally compact group so that $C^*(G)$ is residually finite dimensional, then the algebra $B_\pi$ produced in the proof of Proposition \ref{prop:converse} is equal to $B(G)$ and, so, $B(G)$ is operator Connes-amenable.
\end{remark}

\begin{corollary}
	Suppose $G$ is a discrete group with property (T). If $G$ has the factorization property, then $B(G)$ admits a weak*-closed subalgebra $B$ containing $A(G)$ that is operator Connes-amenable.
\end{corollary}

\begin{proof}
	Discrete property (T) groups with the factorization property are residually finite, hence maximally almost periodic.
\end{proof}

\section{Closing Remarks}\label{sec:end}

Although we now know many examples of non-amenable groups $G$ where $B(G)$ is not operator Connes-amenable, the question of understanding exactly when $B(G)$ is operator Connes-amenable remains wide open. In particular, there are many basic examples of non-amenable groups $G$ where it is unknown whether $B(G)$ is operator Connes-amenable or not.

\begin{prob}
	Is $B(\tn{SL}(2,\R))$ operator Connes-amenable?
\end{prob}

\begin{prob}
	Is $B(G)$ operator Connes-amenable for any non-amenable connected locally compact group $G$?
\end{prob}

Next consider the residually finite group $G=\tn{SL}(3,\Z)$. We know by Proposition \ref{prop:converse} that $B(G)$ admits a weak*-closed subalgebra $B$ containing $A(G)$ that is weak*-closed, however it is unknown whether $B(G)$ itself is operator Connes-amenable. Indeed, since $C^*(G)$ is not residually finite by work of Bekka \cite{Bek}, the algebra $B_\pi$ produced in the proof of Proposition \ref{prop:converse} is strictly contained in $B(G)$.

\begin{prob}
	Is $B(\tn{SL}(3,\Z))$ operator Connes-amenable?
\end{prob}

More generally, one can ask whether the following strong converse to Theorem \ref{thm:discrete-main} holds.

\begin{prob}\label{prob:discrete-main}
	Is $B(G)$ operator Connes-amenable for every discrete group $G$ with the factorization property?
\end{prob}

It is known that a discrete group $G$ has the factorization property if and only if the canonical trace $\delta_e$ of $C^*(G)$ is an amenable trace (e.g., see \cite{Brown} and definitions therein). As such, one can view the factorization property as being a weak form of amenability for $G$ when looking through the correct lens. Problem \ref{prob:discrete-main} having a positive solution would help solidify the viewpoint that the factorization property should be viewed as a weak form of amenability more generally.

An obvious approach when attempting to show that Problem \ref{prob:discrete-main} has a positive solution is to take the positive definite function $1_\Delta\in B_{\min}(G\times G):=[C^*(G)\otimes_{\min}C^*(G)]^*$ for a discrete group $G$ with the factorization property and attempt to extend it to an element of $C_{\sigma,w}([B(G)\widehat\otimes B(G)]^*)^*=C_{\sigma,w}(W^*(G)\bar\otimes W^*(G))^*$ in a way that produces an operator $C_{\sigma,w}$-diagonal. Here $W^*(G)$ denotes the von Neumann algebra $C^*(G)^{**}$. The authors of this paper have been unsuccessful with this approach, perhaps due to lacking a proper understanding of what the space $C_{\sigma,w}(\tn W^*(G)\bar{\otimes}W^*(G))$ consists of. If the following problem were to have a positive solution, a positive solution to Problem \ref{prob:discrete-main} would immediately follow.

\begin{prob}\label{prob:discrete-subprob}
	Is $C_{\sigma,w}(\tn W^*(G)\bar{\otimes}W^*(G))=C^*(G)\otimes_{\min}C^*(G)$ for every discrete group $G$?
\end{prob}

\begin{remark}
	Suppose $G$ is a discrete group with the factorization property and for which $C_{\sigma,w}(\tn W^*(G)\bar{\otimes}W^*(G))=C^*(G)\otimes_{\min}C^*(G)$. Then Problem \ref{prob:discrete-main} has a positive solution for $G$ since $1_\Delta\in B_{\min}(G\times G)=C_{\sigma,w}(W^*(G)\bar\otimes W^*(G))^*$ would clearly be an operator $C_{\sigma,w}$-virtual diagonal for $B(G)$.
\end{remark}

As some evidence that Problem \ref{prob:discrete-subprob} has a hope of being true, we show that it holds when $G$ is amenable.

\begin{prop}\label{prop:disc-amen}
	Let $G$ be an amenable discrete group. Then
	$$C_{\sigma,w}(\tn W^*(G)\bar{\otimes}W^*(G))=C^*(G)\otimes_{\min} C^*(G).$$
\end{prop}

\begin{proof}
	Let $\varpi$ denote the universal representation of $G$ and suppose that $u_1,u_2\in B(G)$ have finite supports $\{s_1,\ldots,s_m\}\subset G$ and $\{t_1,\ldots,t_n\}\subset G$ and $a\in W^*(G)\bar{\otimes}W^*(G)$. Set $b=\sum_{i,j} c_{i,j}\big(\varpi(s_i)\otimes \varpi(t_j)\big)$ where $c_{i,j}=u_1(s_i)u_2(t_j)\lla\delta_{s_i}\otimes \delta_{t_j},a\rra$. We claim that $u_2\cdot a\cdot u_1=b$. Indeed, let $v_1,v_2\in B(G)$. Then
	\begin{eqnarray*}
		&& \lla v_1\otimes v_2,u_2\cdot a\cdot u_1\rra\\
		&=& \lla (u_1v_1)\otimes (u_2v_2), a\rra\\
		&=& \sum_{ij} u(s_i)v_1(s_i)u_2(t_j)v_2(t_j)\lla \delta_{s_i}\otimes \delta_{t_j},a\rra\\
		&=& \sum_{i,j} v_1(s_i)v_2(t_j) c_{i,j}\\
		&=& \lla v_1\otimes v_2,b\rra.
	\end{eqnarray*}
	This verifies our claim.
	
	Now suppose that $a\in C_{\sigma,w}(W^*(G)\bar\otimes W^*(G))$. Since $G$ is amenable, there is a net $\{v_i\}\subset B(G)$ of finitely supported elements converging to $1_G$ in the weak* topology. Then
	$$ \lim_i\lim_j \big(v_j\cdot a\cdot v_i\big)= \lim_i \big(a\cdot v_i\big)= a$$
	in the weak topology due to our assumption that $a\in C_{\sigma,w}(W^*(G)\overline\otimes W^*(G))$.
	Since $C^*(G)\otimes_{\min} C^*(G)$ is closed in the weak topology and $v_j\cdot a\cdot v_i\in C^*(G)\otimes_{\min}C^*(G)$ for all $G$ by our previous claim, we deduce that $a\in C^*(G)\otimes_{\min}C^*(G)$. It follows that $C_{\sigma,w}(\tn W^*(G)\overline{\otimes}W^*(G))=C^*(G)\otimes_{\min}C^*(G)$.
\end{proof}

Recall that every dual Banach algebra admitting a normal, virtual diagonal is Connes-amenable (see \cite[Section 5.1]{Run-book} for the definition and an exposition on the topic). Although the converse to this statement is false \cite{Run4}, it does hold for certain classes of dual Banach algebras, including von Neumann algebras \cite{Co} and measure algebras \cite{Run3}.
The definition of a normal, virtual diagonal can be easily modified to the operator space setting in order to produce the notion of an ``operator normal, virtual diagonal'' for dual, completely contractive Banach algebras. Such algebras admitting an ``operator normal, virtual diagonal'' will necessarily be operator Connes-amenable. Given the fact that measure algebras and Fourier-Stieltjes algebras are often viewed as dual to one another, it is natural to ask the following.

\begin{prob}
	Let $G$ be a locally compact group. Is $B(G)$ operator Connes-amenable if and only if $B(G)$ admits an ``operator normal, virtual diagonal''?
\end{prob}

We finish this paper by noting the solution to a question raised in \cite{W}.

\begin{remark}
	\cite[Problem 7.1]{W} asks whether closed subgroups of locally compact groups with the factorization property necessarily have the factorization property. By referring to Example \ref{ex:Sp}, we can see that this question has a negative solution. Indeed, if $\widetilde G$ is the universal covering of $\tn{Sp}(n,1)$, then $\widetilde{G}$ has the factorization property since $C^*(\widetilde G)$ is nuclear,  but lattices inside of $\widetilde G$ do not have the factorization property.
\end{remark}

\bibliographystyle{amsplain}

\end{document}